\documentclass{amsart}
\usepackage[margin=2.5cm]{geometry}
\usepackage{amsmath,amsfonts,amssymb,amsthm}
\usepackage{graphics}
\usepackage{enumerate}   

\usepackage{epsfig, esint}
\usepackage{graphics,color}
\providecommand{\R}{\mathbb{R}}

\providecommand{\ee}{{\rm e}}

\newcommand{\step}[1]{\medskip\noindent\textbf{Step #1. }}
\newcommand{\substep}[1]{\medskip\noindent\textit{Substep #1. }}

\newcommand{\ignore}[1]{}

\newtheorem{proposition}{Proposition}
\newtheorem{theorem}{Theorem}
\newtheorem{lemma}{Lemma}
\newtheorem{corollary}{Corollary}

\theoremstyle{definition}
\newtheorem{definition}{Definition}
\newtheorem{remark}{Remark}
\newtheorem{assumption}{Assumption}

\author[P.~Bella]{Peter Bella}
\address{Mathematisches Institut 
 Universit\"at Leipzig\\
 Leipzig, 04103 Germany.}
\email{bella@math.uni-leipzig.de}
\author[M. Sch\"affner]{Mathias Sch\"affner}
\address{Mathematisches Institut 
 Universit\"at Leipzig\\
 Leipzig, 04103 Germany.}
\email{schaeffner@math.uni-leipzig.de}

\title{Quenched invariance principle for random walks among random degenerate conductances}

\begin{document}
\maketitle

\begin{abstract}
We consider the random conductance model in a stationary and ergodic environment. Under suitable moment conditions on the conductances and their inverse, we prove a quenched invariance principle for the random walk among the random conductances. The moment conditions improve earlier results of Andres, Deuschel and Slowik [Ann.\ Probab.] and are the minimal requirement to ensure that the corrector is sublinear everywhere. The key ingredient is an essentially optimal deterministic local boundedness result for finite difference equations in divergence form.

\medskip

\noindent
{\bf Keywords:} Random conductance model, invariance principle, stochastic homogenization, non-uniformly elliptic equations.
\end{abstract}


\section{Introduction}

\subsection{Setting of the problem and the main result}

In this paper we study the nearest neighbor random conductance model on the $d$-dimensional Euclidean lattice $(\mathbb Z^d,\mathbb B^d)$, for $d\geq3$. Here $\mathbb B^d$ is given by the set of nonoriented nearest neighbor bounds that is $\mathbb B^d:=\{\{x,y\}\,|\,x,y\in\mathbb Z^d,\,|x-y|=1\}$. 

We set $\Omega:=(0,\infty)^{\mathbb B^d}$ and call $\omega(\ee)$ the \textit{conductance} of the bond $\ee\in\mathbb B^d$ for every $\omega=\{\omega(\ee)\,|\,\ee\in \mathbb B^d\}\in\Omega$. To lighten the notation, for any $x,y\in\mathbb Z^d$, we set
\begin{equation*}
\omega(x,y)=\omega(y,x):=\omega(\{x,y\})\quad\forall \{x,y\}\in\mathbb B^d,\qquad \omega(\{x,y\})=0\quad\forall \{x,y\}\notin\mathbb B^d.
\end{equation*}
In what follows we consider random conductances that are distributed according to a probability measure $\mathbb P$  on $\Omega$ equipped with the $\sigma$-algebra $\mathcal F:=\mathcal B((0,\infty))^{\otimes \mathbb B^d}$ and we write $\mathbb E$ for the expectation with respect to $\mathbb P$. 

We introduce the family of \textit{space shifts} $\{\tau_x:\Omega\to\Omega\,|\,x\in\mathbb Z^d\}$ defined by
\begin{equation*}
 \tau_x\omega(\cdot):=\omega(\cdot+x)\qquad\mbox{where for any $\ee=\{\underline \ee,\overline\ee\}\in\mathbb B^d$, $\ee+x:=\{\underline e+x,\overline\ee+x\}\in\mathbb B^d$.}
\end{equation*}
For any fixed realization $\omega$, we study the reversible continuous time Markov chain, $X = \{X_t : t \geq0\}$, on $\mathbb Z^d$ with generator $\mathcal L^\omega$ acting on bounded functions $f :\mathbb Z^d\to\mathbb R$ as
\begin{equation}\label{def:Lomega}
 (\mathcal L^\omega f)(x):=\sum_{y\in\mathbb Z^d}\omega(x,y)(f(y)-f(x)).
\end{equation}
We emphasize at this point that $\mathcal L^\omega$ is in fact a finite-difference operator in divergence form, see \eqref{Lomegadiv} below. Following \cite{ADS15}, we denote by ${\bf P}_x^\omega$ the law of the process starting at the vertex $x\in\mathbb Z^d$ and by ${\bf E}_x^\omega$ the corresponding expectation. $X$ is called the \textit{variable speed random walk} (VSRW) since it waits at $x\in\mathbb Z^d$ an exponential time with mean $1/\mu^\omega(x)$, where $\mu^\omega(x)=\sum_{y\in\mathbb Z^d}\omega(x,y)$ and chooses its next position $y$ with probability $p^\omega(x,y):=\omega(x,y)/\mu^\omega(x)$.

\begin{assumption}\label{ass}
Assume that $\mathbb P$ satisfies the following conditions
\begin{enumerate}[(i)]
 \item (stationarity) $\mathbb P$ is stationary with respect to shifts, that is $\mathbb P\circ \tau_x^{-1}=\mathbb P$ for all $x\in\mathbb Z^d$.
 \item (ergodicity) $\mathbb P$ is ergodic, that is $\mathbb P[A]\in\{0,1\}$ for any $A\in\mathcal F$ such that $\tau_x(A)=A$ for all $x\in\mathbb Z^d$
 \item (moment condition) There exists $p,q\in(1,\infty]$ satisfying 
 \begin{equation}\label{eq:pq}
 \frac1p+\frac1q<\frac2{d-1}
 \end{equation}
 such that
 \begin{equation}\label{ass:moment}
 \mathbb E[\omega(\ee)^p]<\infty,\quad \mathbb E[\omega(\ee)^{-q}]<\infty\qquad\mbox{for any $\ee\in\mathbb B^d$}.
 \end{equation}
\end{enumerate}
\end{assumption}

The main result of the present paper is a quenched invariance principle for the process $X$ under Assumption~\ref{ass}.
\begin{definition}\label{def:qfclt}
Set $X_t^{(n)}:=\frac1n X_{n^2t}$, $t\geq0$. We say that a \textit{quenched functional CLT} (QFCLT) or \textit{quenched invariance principle} holds for $X$ if for $\mathbb P$-a.e.\ $\omega$ under ${\bf P}_0^\omega$, $X^{(n)}$ converges in law to a Brownian motion on $\R^d$ with covariance matrix $\Sigma^2=\Sigma\cdot \Sigma^t$. That is, for every $T>0$ and every bounded continuous function $F$ on the Skorokhod space $D([0,T],\R^d)$, setting $\psi_n={\bf E}_0^\omega[F(X^{(n)})]$ and $\psi_\infty={\bf E}_0^{\rm BM}[F(\sigma\cdot W)]$ with $(W,{\bf P}_0^{\rm BM})$ being a Brownian motion started at $0$, we have that $\psi_n\to\psi_\infty$, $\mathbb P$-a.s.
\end{definition}

\begin{theorem}[Quenched invariance principle]\label{T}
 Suppose $d\geq3$ and that Assumption~\ref{ass} is satisfied. Then the QFCLT holds for $X$ with a deterministic nondegenerate covariance matrix $\Sigma^2$. 
\end{theorem}

\begin{remark}\label{rem:CSRW}
Another natural process is given by the so called \textit{constant speed random walk (CSRW)} $Y$ which is defined via the generator $\mathcal L_Y^\omega$ 
 \begin{equation*}
 (\mathcal L_Y^\omega f)(x):=\frac1{\mu^\omega(x)}\sum_{y\in\mathbb Z^d}\omega(x,y)(f(y)-f(x)),
  \end{equation*}
where $\mu^\omega(x)=\sum_{y\in\mathbb Z^d}\omega(x,y)$. In contrast to the VSRW the CSRW waits on each vertex $x\in\mathbb Z^d$ an exponential time with mean $1$.  The invariance principle for the VSRW $X$ and Assumption~\ref{ass} imply also a QFCLT for $Y$ with a covariance matrix given by $\mathbb [\mu^\omega(0)]^{-1}\Sigma^2$ (where $\Sigma$ is as in Theorem~\ref{T}), see \cite[Remark~1.5]{ADS15}.
\end{remark}

Random walks among random conductances are one of the most studied examples of random walks in a random environment, see \cite{Biskup,Kumagai} for relatively recent overviews of the field. In \cite{DMFGW89} (see also \cite{KV86}) a weak FCLT, that is the convergence of $\psi_n$ to $\psi_\infty$ in Definition~\ref{def:qfclt} holds in $\mathbb P$-probability, for stationary and ergodic laws $\mathbb P$ with $\mathbb E[\omega(\ee)] < \infty$ is established. In the last two decades much attention has been devoted to obtain quenched FCLT. In \cite{SS04}, the quenched invariance principle is proven in the uniformly elliptic case, i.e.\  with the assumption that there exists $c\in(0,1]$ such that $\mathbb P[c\leq \omega(\ee)\leq c^{-1}])=1$ for all $\ee\in \mathbb B^d$, which corresponds to the case $p=q=\infty$ (see also an earlier result \cite{B93} valid only in $d=2$). Recently there is an increasing interest to relax the uniform ellipticity assumption. In the special case of i.i.d.\ conductances, that is when $\mathbb P$ is the product measure which includes e.g.\ percolation models, it is shown in \cite{ABDH13} (building on previous works \cite{BD10,BB07,M08,MP07,SS04}) that a QFCLT holds provided that $\mathbb P[\omega(\ee) > 0] > p_c$ with $p_c=p_c(d)$ being the bond percolation threshold. In particular no moment conditions such as \eqref{ass:moment} are needed. In the general ergodic situation it is known that at least first moments of $\omega$ and $\omega^{-1}$ are necessary for a QFCLT to hold (see \cite{BBT16} for an example where the QFCLT fails but \eqref{ass:moment} holds for any $p,q\in(0,1)$). In \cite{ADS15},  Andres, Deuschel and Slowik proved the conclusion of Theorem~\ref{T} under the moment condition \eqref{ass:moment} with the more restrictive relation  
\begin{equation}\label{eq:pqold}
\frac1p+\frac1q<\frac2{d}.
\end{equation}
The result of \cite{ADS15} was already extended in several directions: to the continuum case \cite{CD16} (for an earlier contribution with $q=\infty$, see \cite{FK97}), random walks on more general graphs \cite{DNS} and to dynamic situations \cite{ACDS18}, see also \cite{BR18,MO16}. Previous to \cite{ADS15}, Biskup \cite{Biskup} proved QFCLT under the minimal moment condition $p=q=1$ in two dimensions and thus we focus our attention to the case $d\geq3$. To the best of our knowledge Theorem~\ref{T} is the first quenched invariance principle in the general stationary \& ergodic setting under less restrictive moment condition compared to \eqref{eq:pqold} valid in $d\geq3$. Optimality of condition \eqref{eq:pq} in Theorem~\ref{T} is not clear to us, since in particular in \cite{BM15} a quenched invariance principle for diffusion in $\R^d$ with a locally integrable periodic potential is proven. However, we emphasize that condition \eqref{eq:pq} is essentially optimal for the everywhere sublinearity of the corrector, see Proposition~\ref{P:sublinlinfty} and Remark~\ref{rem:sublin}. The latter is of independent interest for stochastic homogenization of elliptic operators in divergence form with degenerate coefficients, for further recent results in that direction, see \cite{AD18,AN17,BFO18,NSS17}.

\subsection{Strategy}
The proof of Theorem~\ref{T} follows the classical approach to show an invariance principle and relies on a decomposition of the process $X$ into a martingale part and a remainder (see e.g.\ \cite{KV86}). General martingale theory (in particular \cite{Helland}) yields a quenched invariance principle for the martingale part and it remains to show that the remainder is negligible. A key insight in \cite{ADS15} was to apply deterministic elliptic regularity theory, in particular Moser's iteration argument \cite{Moser60,Moser61}, to control the remainder term. The main effort in the present contribution is to improve the deterministic part of the argument. Let us now be more precise (in what follows we use the notation introduced in Section~\ref{sec:notation} below). Following e.g.\ \cite{ADS15,Biskup}, we introduce harmonic coordinates, that is, we construct a corrector field $\chi:\Omega\times\mathbb Z^d\to\R^d$ such that
\begin{equation*}
 \Phi(\omega,x)=x-\chi(\omega,x)
\end{equation*} 
is $\mathcal L^\omega$-harmonic in the sense that for every $x\in\mathbb Z^d$ and $j\in\{1,\dots,d\}$ 
\begin{equation}\label{eq:correctointro}
 0=\mathcal L^\omega(\Pi_j-\chi_j)(x)=-\nabla^*(\omega\nabla(\Pi_j-\chi_j))(x),
\end{equation}
where $\Pi_j(y)=y\cdot e_j$ and $\chi_j(y)=\chi(y)\cdot e_j$ for every $y\in\mathbb Z^d$. The $\mathcal L^\omega$-harmonicity of $\Phi$  implies that
\begin{equation*}
 M_t:=\Phi(\omega,X_t)=X_t-\chi(\omega,X_t)
\end{equation*}
is a martingale under ${\bf P}_0^\omega$ for $\mathbb P$-a.e.\ $\omega$. The QFCLT of $M$ can e.g.\ be found in \cite{ADS15} under less restrictive assumptions compared to Assumption~\ref{ass}, see Proposition~\ref{P:propprevious} below. In order to establish the QFCLT for $X$, we show that for any $T>0$ and $\mathbb P$-a.e.\ $\omega$
\begin{equation*}
\sup_{t\in[0,T]}\frac1n|\chi(\omega,nX_t^{(n)})|\to0\quad\mbox{in ${\bf P}_0^\omega$-probability as $n\to\infty$}
\end{equation*}
see Proposition~\ref{P:correctionvanishes} below. In fact, we establish a much stronger statement: instead of proving sublinearity of $\chi$ along the path of the process $X$ we show sublinearity everywhere
\begin{equation}\label{eq:sublinintro}
\lim_{n\to\infty}\max_{x\in B(n)}\frac1n|\chi(\omega,x)|=0\qquad\mbox{for $\mathbb P$-a.e.\ $\omega$}
\end{equation}
see Proposition~\ref{P:sublinlinfty} below. The proof of \eqref{eq:sublinintro} relies on the following deterministic regularity result for $\mathcal L^\omega$-harmonic functions 
\begin{theorem}\label{T1}
Fix $d\geq3$, $\omega\in\Omega$ and let $p,q\in(1,\infty]$ be such that $\frac1p+\frac1q<\frac2{d-1}$. Then there exists $c=c(d,p,q)\in[1,\infty)$ such that solutions of $\nabla^*({\omega}\nabla u)=0$ in $\mathbb Z^d$ satisfy for every $y\in\mathbb Z^d$ and every $n\in \mathbb N$
\begin{equation}\label{est:T:boundl1}
 \max_{x\in B(y,n)}|u(x)|\leq c\Lambda^\omega( B(y,2n))^{p'\frac{\delta+1}{\delta}}\|u\|_{\underline L^{1}(B(y,2n))},
\end{equation}
where $\delta:=\frac{1}{d-1}-\frac1{2p}-\frac1{2q}>0$, $p':=\frac{p}{p-1}$ and for every bounded set $ S\subset\mathbb Z^d$ 
\begin{equation}\label{def:lambda}
\Lambda^\omega(S):=\|\omega\|_{\underline L^p(S)}\|\omega^{-1}\|_{\underline L^q(S)}.
\end{equation}
\end{theorem}
\begin{remark}
A continuum version of Theorem~\ref{T1} was recently proven by the authors of the present paper in \cite{BS19a}. In the continuum case relation $\frac1p+\frac1q<\frac2{d-1}$ is essentially optimal for local boundedness (see \cite{FSS98}) and so it is in the discrete setting considered here, see Remark~\ref{rem:sublin} below. In \cite{ADS15} a version of Theorem~\ref{T1} is proven for solutions of the Poisson equation
\begin{equation}\label{eq:poisson}
\nabla^*(\omega \nabla u)=\nabla^*(\omega\nabla f)
\end{equation}
on rather general weighted graphs but under the more restrictive relation $\frac1p+\frac1q<\frac2d$, see \cite[Theorem~3.7]{ADS15} (for related classical results in the continuum see \cite{MS68,T71,T73}). This regularity statement is then applied in \cite{ADS15} to the corrector equation \eqref{eq:correctointro} to ensure \eqref{eq:sublinintro}. Our method does not directly apply to solutions of \eqref{eq:poisson} but due to the specific form of the right-hand side in the corrector equation \eqref{eq:correctointro}, i.e.\  $f(x)=x\cdot e_j$, we are able to deduce from Theorem~\ref{T1} the needed sublinearity of the corrector.
\end{remark}

\begin{remark}
 In \cite{BS19a}, we also establish Harnack inequality for non-negative solutions $u$ and we expect that this can be extended to the discrete case, too. In \cite{ADS16}, Andres, Deuschel and Slowik establish elliptic and parabolic versions of Harnack inequality for the CSRW, see Remark~\ref{rem:CSRW}, on weighted graphs under moment conditions \eqref{ass:moment} with $\frac1p+\frac1q<\frac2d$. From the parabolic version they deduced a quenched local limit theorem and showed that condition $\frac1p+\frac1q<\frac2d$ is essentially optimal for that result. It is an interesting question if the methods developed here can be used to derive parabolic Harnack inequality and local limit theorems for the VSRW under less restrictive relations between the exponents $p$ and $q$ compared to the CSRW. 
\end{remark}

\subsection{Notation}\label{sec:notation}

\begin{itemize}
\item (Sets and $L^p$ spaces) For $y\in\mathbb Z^d$, $n\in\mathbb N$, we set $B(y,n):=y+([-n,n]\cap\mathbb Z)^d$ with the shorthand $B(n)=B(0,n)$. For any $S\subset \mathbb Z^d$ we denote by $S_{\mathbb B^d}\subset \mathbb B^d$ the set of bonds for which both end-points are contained in $S$, i.e.\ $S_{\mathbb B^d}:=\{\ee=\{\underline \ee,\overline \ee\}\in\mathbb B^d\,|\, \underline \ee,\overline\ee \in S\}$. For any $S\subset\mathbb Z^d$, we set $\partial S:=\{x\in S\;|\; \exists y\in\mathbb Z^d\setminus S\mbox{ s.t. } \{x,y\}\in\mathbb B^d\}$. Given $p\in(0,\infty)$, $S\subset \mathbb Z^d$, we set for any $f:\mathbb Z^d\to\R^d$ and $F:\mathbb B^d\to\R$ 
\begin{equation*}
 \|f\|_{L^p(S)}:=\left(\sum_{x\in S}|f(x)|^p\right)^\frac1p,\quad \|F\|_{L^p(S_{\mathbb B^d})}:=\left(\sum_{\ee\in S_{\mathbb B^d}}|F(\ee)|^p\right)^\frac1p,
\end{equation*}
and $\|f\|_{L^\infty(S)}=\sup_{x\in S}|f(x)|$. Moreover, normalized versions of $\|\cdot||_{L^p}$ are defined for any finite subset $S\subset \mathbb Z^d$ and $p\in(0,\infty)$ by
\begin{equation*}
 \|f\|_{\underline L^p(S)}:=\left(\frac1{|S|}\sum_{x\in S}|f(x)|^p\right)^\frac1p,\quad \|F\|_{\underline L^p(S_{\mathbb B^d})}:=\left(\frac1{|S_{\mathbb B^d}|}\sum_{\ee\in S_{\mathbb B^d}}|F(\ee)|^p\right)^\frac1p,
\end{equation*}
where $|S|$ and $|S_{\mathbb B^d}|$ denote the cardinality of $S$ and $S_{\mathbb B^d}$, respectively. Throughout the paper we drop the subscript in $S_{\mathbb B^d}$ if the context is clear.
\item (discrete calculus) For any bond $\ee\in\mathbb B^d$, we denote by $\underline \ee,\overline \ee\in\mathbb Z^d$ the (unique) vertices satisfying $\ee=\{\underline \ee,\overline \ee\}$ and $\overline \ee - \underline \ee\in\{e_1,\dots,e_d\}$. For $f:\mathbb Z^d\to\R$, we define its \textit{discrete derivative} as
\begin{equation*}
 \nabla f:\mathbb B^d\to\R,\qquad \nabla f(\ee):=f(\overline \ee)-f(\underline \ee).
\end{equation*} 
For $f,g:\mathbb Z^d\to\R$ the following discrete product rule is valid
\begin{align}\label{chainrule}
 \nabla (fg)(\ee)=f(\overline \ee)\nabla g(\ee)+g(\underline \ee)\nabla f(\ee)=f(\ee)\nabla g(\ee)+g(\ee)\nabla f(\ee),
\end{align} 
where we use for the last equality the convenient identification of a function $h:\mathbb Z^d\to\R$ with the function $h:\mathbb B^d\to \R$ defined by the corresponding arithmetic mean
\begin{equation*}
 h(\ee):=\frac12 (h(\overline \ee)+h(\underline \ee)).
\end{equation*}
The \textit{discrete divergence} is defined for every $F:\mathbb B^d\to\R$ as
\begin{equation*}
\nabla^*F(x):=\sum_{\ee\in\mathbb B^d\atop \overline \ee=x}F(\ee)-\sum_{\ee\in\mathbb B^d\atop \underline \ee=x}F(\ee)=\sum_{i=1}^d\left(F(\{x-e_i,x\})-F(\{x,x+e_i\})\right).
\end{equation*}
Note that for every $f:\mathbb Z^d\to\R$ that is non-zero only on finitely many vertices and every $F:\mathbb B^d\to\R$ it holds
\begin{equation}\label{sumbyparts}
 \sum_{\ee\in\mathbb B^d}\nabla f(\ee)F(\ee)=\sum_{x\in\mathbb Z^d}f(x)\nabla^*F(x).
\end{equation}
Finally, we observe that the generator $\mathcal L^\omega$ defined in \eqref{def:Lomega} can be written as a second order finite-difference operator in divergence form, in particular
\begin{equation}\label{Lomegadiv}
\forall u:\mathbb Z^d\to\R\qquad \mathcal L^\omega u(x)=-\nabla^*(\omega\nabla u)(x)\quad\mbox{for all $x\in\mathbb Z^d$.}
\end{equation}

\end{itemize}

\section{The quenched invariance principle}

In this section we proof Theorem~\ref{T}. As mentioned above we follow a well established strategy and decompose the process $X$ such that $M_t=X_t-\chi(\omega,X_t)$ is a martingale under ${\bf P}_0^\omega$ for $\mathbb P$-a.e.\ $\omega$. It is already known that under Assumption~\ref{ass} the martingale part $M$ satisfies a QFCLT and it is left to show that the remainder $\chi(\omega,X_t)$ vanishes in a suitable sense. In Section~\ref{subsec:old}, we recall the construction of the corrector from \cite{ADS15} and state the needed known results for $M$ and $\chi$. In Section~\ref{subsec:sublin}, we use Theorem~\ref{T1} to prove that the corrector is sublinear everywhere.

\subsection{Harmonic embedding and the corrector}\label{subsec:old}

The construction of the corrector and the invariance principle for the martingale part can be found in the literature, see e.g.\ \cite{ADS15,Biskup}. For convenience we recall the needed results
\begin{definition}
A random field $\Psi:\Omega\times \mathbb Z^d\to\R$ satisfies the \textit{cocycle property} if for $\mathbb P$-a.e.\ $\omega$
\begin{equation*}
 \Psi(\tau_x\omega,y-x)=\Psi(\omega,y)-\Psi(\omega,x)\qquad\mbox{for all $x,y\in\mathbb Z^d$}.
\end{equation*}
We denote by $L^2_{\rm cov}$ the set of functions $\Psi:\Omega\times \mathbb Z^d\to\R$ satisfying the cocycle property such that
\begin{equation*}
\|\Psi\|_{L^2_{\rm cov}}^2:=\mathbb E\biggl[\sum_{x\in\mathbb Z^d}\omega(0,x)\Psi(\omega,x)^2\biggr]<\infty. 
\end{equation*}
\end{definition}
Note that
\begin{lemma}
 $L^2_{\rm cov}$ is a Hilbert-space.  
\end{lemma}
A function $\phi:\Omega\to\R$ is called \textit{local} if it depends on the value of $\omega\in\Omega$ (recall $\Omega=(0,\infty)^{\mathbb B^d}$) at finitely many bonds $\ee\in\mathbb B^d$. The horizontal derivative $D\phi:\Omega\times\mathbb Z^d\to\R$ of $\phi$ is defined by
\begin{equation*}
D\phi(\omega,x)=\phi(\tau_x\omega)-\phi(\omega),\qquad x\in\mathbb Z^d.
\end{equation*}
We define the subspace $L^2_{\rm pot}$ of potential random fields as
\begin{equation*}
 L^2_{\rm pot}:={\rm cl}\{D\phi\,|\,\phi:\Omega\to\R\,\mbox{local}\}\quad\mbox{in $L^2_{\rm cov}$}
\end{equation*}
and the subspace $L_{\rm sol}^2$, of solinoidal random fields, as the orthogonal complement of $L_{\rm pot}^2$ in $L_{\rm cov}^2$.

The corrector is now constructed as a suitable projection. For this we introduce the position field $\Pi:\Omega\times \mathbb Z^d\to\R^d$ with $\Pi(\omega,x)=x$ for all $x\in\mathbb Z^d$ and $\omega\in\Omega$. Set $\Pi_j:=\Pi\cdot e_j$ and observe that $\Pi_j$ satisfies the cocycle property and $ \|\Pi_j\|_{L_{\rm cov}^2}^2=2\mathbb E[\omega(0,e_j)]<\infty$. Hence, $\Pi_j\in L^2_{\rm cov}$ and we define $\chi_j\in L_{\rm pot}^2$ and $\Phi_j\in L_{\rm sol}^2$ by 
\begin{equation}\label{eq:defchij}
\Pi_j=\chi_j+\Phi_j\in L_{\rm pot}^2\oplus L_{\rm sol}^2.
\end{equation} 
Finally, we define the corrector $\chi=(\chi_1,\dots,\chi_d):\Omega\times\mathbb Z^d\to \R^d$ and the process $M_t$ as
\begin{equation*}
 M_t:=\Phi(\omega,X_t)=X_t-\chi(\omega,X_t).
\end{equation*}
The needed properties of $M_t$, $\Phi$ and $\chi$ are gathered in the following
\begin{proposition}[\cite{ADS15}]\label{P:propprevious}
 Let $d\geq2$ and suppose that part (i) and (ii) of Assumption~\ref{ass} are satisfied. Moreover, suppose that $\mathbb E[\omega(\ee)]<\infty$ and $\mathbb E [\omega(\ee)^{-1}]<\infty$ for every $\ee\in \mathbb B^d$. Then there exists $\Omega_1\subset\Omega$ with $\mathbb P(\Omega_1)=1$ such that,
 \begin{enumerate}[(i)]
 \item ($\mathcal L^\omega$-harmonicity of $\Phi$) for all $\omega\in\Omega_1$
 \begin{equation}\label{eq:phiharmonic}
 \mathcal L^\omega\Phi(x)=\sum_{y\in\mathbb Z^d}\omega(x,y)(\Phi(\omega,y)-\Phi(\omega,x))=0\in\R^d, \qquad \Phi(\omega,0)=0\in\R^d. 
 \end{equation}
 \item (QFCLT for $M$) Set $M_t^{(n)}:=\frac1n M_{n^2t}$, $t\geq0$. For all $\omega\in\Omega_1$, the sequence $\{M^{(n)}\}$ converges in law in the Skorokhod topology to a Brownian motion with a nondegenerate covariance matrix $\Sigma^2$ given by
 \begin{equation*}
  \Sigma_{ij}^2=\mathbb E\biggl[\sum_{x\in\mathbb Z^d}\omega(0,x)\Phi_i(\omega,x)\Phi_j(\omega,x)\biggr].
 \end{equation*}
 \item ($L^1$-sublinearity of $\chi$) For all $\omega\in\Omega_1$ and any $j\in\{1,\dots,d\}$,
 \begin{equation}\label{eq:sublinl1}
 \lim_{n\to\infty}\frac1n\|\chi_j(\omega,\cdot)\|_{\underline L^1(B(n))}=0.
 \end{equation}
 \end{enumerate}
\end{proposition}
Statement (i) is contained in \cite[Proposition~2.3]{ADS15}, (ii) is contained \cite[Proposition~2.5]{ADS15}, and (iii) can be easily deduced from \cite[Proposition~2.9]{ADS15}.

\subsection{$L^\infty$-sublinearity of the corrector}\label{subsec:sublin}

In this section we improve the sublinearity of the corrector in $L^1$, see Proposition~\ref{P:propprevious} part~(iii), to sublinearity in the $L^\infty$-sense. This is content of the following
\begin{proposition}\label{P:sublinlinfty}
Let $d\geq3$ and suppose that Assumptions~\ref{ass} is satisfied. Then, for any $L\in\mathbb N$ and $j\in\{1,\dots,d\}$ 
\begin{equation}\label{sublinerlinfty}
\lim_{n\to\infty}\frac1n\max_{x\in B(Ln)}|\chi_j(\omega,x)|=0\qquad\mbox{$\mathbb P$-a.s.}
\end{equation}
\end{proposition}

\begin{remark}\label{rem:sublin}
In \cite{ADS15}, the sublinearity of the corrector in the form \eqref{sublinerlinfty} is shown under moment conditions \eqref{ass:moment} with the more restrictive relation $\frac1p+\frac1q<\frac2d$. In two dimensions \eqref{sublinerlinfty} is proven in \cite{Biskup} under the minimal assumptions $p=q=1$ and thus we focus here on $d\geq3$ (see however Section~\ref{sec:2d} for a discussion of the case $d=2$). We emphasize that Assumption~\ref{ass} is essentially optimal for the conclusion of Proposition~\ref{P:sublinlinfty}. Indeed, it was recently shown by Biskup and Kumagai \cite{BK14} that the statement of Proposition~\ref{P:sublinlinfty} fails if \eqref{ass:moment} only holds for $p,q$ satisfying $\frac1p+\frac1q>\frac{2}{d-1}$ provided $d\geq4$, see \cite[Theorem~2.7]{BK14}. This non-existence of a sublinear corrector implies that the condition $\frac1p+\frac1q<\frac2{d-1}$ in Theorem~\ref{T1} is essentially sharp. Indeed, if estimate \eqref{est:T:boundl1} were valid for some $p,q\in[1,\infty]$, then the proof of Proposition~\ref{P:sublinlinfty} together with Proposition~\ref{P:propprevious} yield \eqref{sublinerlinfty} which contradicts the findings in \cite{BK14} if $\frac1{p}+\frac1{q}>\frac2{d-1}$. 
\end{remark}

\begin{proof}[Proof of Proposition~\ref{P:sublinlinfty}]

Throughout the proof we write $\lesssim$ if $\leq$ holds up to a positive constant which depends only on $d,p$ and $q$. Before we give the details of the proof, we briefly explain the idea. We introduce an additional length-scale $\frac{n}{m}$ with $m\in\mathbb N$ such that $1\ll m\ll n$ and compare $\chi_j$ on boxes with diameter $\sim \frac{n}{m}$ with $\mathcal L^\omega$-harmonic functions $\Phi_j-(e_j\cdot x-c)$ with a suitable chosen $c\in \R$. Using the $L^1$-sublinearity of $\chi_j$ and the fact that the linear part coming from $e_j\cdot x$ can be controlled by $\frac1m$ on each box of radius $\sim \frac{n}{m}$ we obtain the desired claim. 

\step 1 As a preliminarily step, we recall the needed input from ergodic theory. Following \cite{ADS15}, we introduce the following measures $\mu^\omega$ and $\nu^\omega$ on $\mathbb Z^d$:
\begin{equation}\label{def:munu}
 \mu^\omega(x)=\sum_{y\in\mathbb Z^d}\omega(x,y)\quad\mbox{and}\quad\nu^\omega(x)=\sum_{y\in\mathbb Z^d}\frac1{\omega(x,y)}.
\end{equation}
In view of the spatial ergodic theorem, we obtain from the moment condition \eqref{ass:moment} that there exists $\Omega'\subset \Omega$ with $\mathbb P(\Omega')=1$  such that for $\omega\in\Omega'$ and every $z\in\mathbb Z^d$
\begin{equation}\label{eq:ergodic}
 \lim_{n\to\infty}\|\mu^\omega\|_{\underline L^p(B(nz,n))}^p=\mathbb E[\mu^\omega(0)^p]<\infty\quad\mbox{and}\quad \lim_{n\to\infty}\|\nu^\omega\|_{\underline L^q(B(nz,n))}^q=\mathbb E[\nu^\omega(0)^q]<\infty,
\end{equation}
see e.g.\ \cite[Section~6]{Krengel}.

\step 2 We set $\Omega_2:=\Omega_1\cap \Omega'$, where $\Omega'$ is given as in Step~1 and $\Omega_1$ in Proposition~\ref{P:propprevious}. Clearly $\Omega_2$ has full measure. From now on we fix $\omega\in\Omega_2$. 

\smallskip

Fix $m\in\mathbb N$. For $n$ sufficiently large compared to $m$ (the choice $n\geq m(m+1)$ will do), we cover the box $B(n)$ with finitely many boxes $B(\lfloor \frac{n}{m}\rfloor z,\lfloor \frac{n}{m}\rfloor)$, $z\in B(m)$. For $z\in B(m)$, set $u_j^z(\omega,x):=\chi_j(\omega,x)-e_j\cdot(x-\lfloor \frac{n}{m}\rfloor z)=-\Phi_j(\omega,x)+e_j\cdot \lfloor \frac{n}{m}\rfloor z$. Obviously, \eqref{eq:phiharmonic} implies that $u_j^z$ is $\mathcal L^\omega$-harmonic. Hence, \eqref{est:T:boundl1} yields
\begin{align}\label{moser:ui}
\|u_j^z\|_{L^\infty(B(\lfloor \frac{n}{m}\rfloor z,\lfloor \frac{n}{m}\rfloor))}\lesssim& \Lambda^\omega(B(\lfloor \tfrac{n}{m}\rfloor z,2\lfloor \tfrac{n}{m}\rfloor))^{p'(1+\frac1\delta)}\|u_j^z\|_{\underline L^1(B(\lfloor \frac{n}{m}\rfloor z,2\lfloor \frac{n}{m}\rfloor))}\notag\\
\lesssim&\Lambda^\omega( B(\lfloor \tfrac{n}{m}\rfloor z,2\lfloor \tfrac{n}{m}\rfloor))^{p'(1+\frac1\delta)}\left(\|\chi_j\|_{\underline L^1(B(\lfloor \frac{n}{m}\rfloor z,2\lfloor \frac{n}{m}\rfloor))}+\lfloor \tfrac{n}m\rfloor\right),
\end{align}
where $p'=\frac{p}{p-1}$, $\delta=\frac1{d-1}-\frac1{2p}-\frac1{2q}>0$. Estimate \eqref{moser:ui} implies the following $L^\infty$-estimate on $\chi_j$
\begin{eqnarray}\label{est:phiexi:1}
 \|\chi_j\|_{L^\infty(B(n))}&\leq& \sup_{z\in B(m)}\|\chi_j\|_{L^\infty(B(\lfloor \frac{n}{m}\rfloor z,\lfloor \frac{n}{m}\rfloor))}\notag\\
 &\lesssim&  \sup_{z\in B(m)}\|u_j^z\|_{L^\infty(B(\lfloor \frac{n}{m}\rfloor z,\lfloor \frac{n}{m}\rfloor))}+\lfloor \tfrac{n}m\rfloor\notag\\
 &\stackrel{\eqref{moser:ui}}{\lesssim}& \sup_{z\in B(m)}\Lambda^\omega( B(\lfloor \tfrac{n}{m}\rfloor z,2\lfloor \tfrac{n}{m}\rfloor))^{p'(1+\frac1\delta)}\left(\|\chi_j\|_{\underline L^1(B(\lfloor \frac{n}{m}\rfloor z,2\lfloor \frac{n}{m}\rfloor))}+\lfloor \tfrac{n}m\rfloor\right)+\lfloor \tfrac{n}m\rfloor\notag\\
 &\lesssim& (m^{d}\|\chi_j\|_{\underline L^1(B(2n)}+\lfloor \tfrac{n}m\rfloor)\sup_{z\in B(m)}\Lambda^\omega( B(\lfloor \tfrac{n}{m}\rfloor z,2\lfloor \tfrac{n}{m}\rfloor))^{p'(1+\frac1\delta)}+\lfloor \tfrac{n}m\rfloor.
\end{eqnarray}
Since $B(m)$ is a finite set, we obtain from the definition of $\mu$ and $\nu$, see \eqref{def:munu}, and the spatial ergodic theorem in the form \eqref{eq:ergodic} that
\begin{align}\label{limsupLambdae}
&\limsup_{n\to\infty}\max_{z\in B(m)}\Lambda^\omega(B(\lfloor \tfrac{n}{m}\rfloor z,2\lfloor \tfrac{n}{m}\rfloor))\notag\\
\lesssim&\limsup_{n\to\infty}\max_{z\in B(m)}\|\mu^\omega\|_{\underline L^p(B(2\lfloor \tfrac{n}{m}\rfloor z,2\lfloor \tfrac{n}{m}\rfloor)}\|\nu^\omega\|_{\underline L^q(B(2\lfloor \tfrac{n}{m}\rfloor z,2\lfloor \tfrac{n}{m}\rfloor)}\notag\\
\leq& \mathbb E[\mu^{p}(0)]^\frac1p\mathbb E[\nu^q(0)]^\frac1q<\infty.
\end{align}
Finally, we combine \eqref{est:phiexi:1} and \eqref{limsupLambdae} with the $L^1$-sublinearity of $\chi_j$ \eqref{eq:sublinl1} and obtain
\begin{equation*}
 \limsup_{n\to\infty}\frac1n\|\chi_j\|_{L^\infty(B(n))}\lesssim m^{-1}( \mathbb E[\mu^{p}(0)]^\frac1p\mathbb E[\nu^q(0)]^\frac1q)^{p'(1+\frac1\delta)}+m^{-1}.
\end{equation*}
The arbitrariness of $m\in\mathbb N$ implies \eqref{sublinerlinfty} for $L=1$ and the trivial identity
$$
\lim_{n\to\infty}\frac1n\max_{x\in B(Ln)}|\chi_j(\omega,x)|=L\lim_{n\to\infty}\frac1n\max_{x\in B(n)}|\chi_j(\omega,x)|=0
$$
finishes the proof.

\end{proof}

\subsection{Proof of Theorem~\ref{T}}

With help of Proposition~\ref{P:propprevious} and \ref{P:sublinlinfty} we can establish Theorem~\ref{T} following the argument in \cite{ADS15}. First, we observe that Proposition~\ref{P:sublinlinfty} implies
\begin{proposition}\label{P:correctionvanishes}
Let $T>0$. For $\mathbb P$-a.e.\ $\omega$,
\begin{equation*}
\sup_{t\in[0,T]}\frac1n|\chi(\omega,nX_t^{(n)})|\to0\quad\mbox{in ${\bf P}_0^\omega$-probability as $n\to\infty$}.
\end{equation*}
\end{proposition}

\begin{proof}[Proof of Proposition~\ref{P:correctionvanishes}]
Appealing to Proposition~\ref{P:sublinlinfty} we can follow verbatim the argument of the proof given \cite[Proposition~2.13]{ADS15}.
\end{proof}

\begin{proof}[Proof of Theorem~\ref{T}]
 A combination of Proposition~\ref{P:propprevious} (part~(ii)) and Proposition~\ref{P:correctionvanishes} yields the desired claim. 
\end{proof}

\section{Local boundedness for $\mathcal L^\omega$-harmonic functions}

\subsection{An auxiliary Lemma}

In this section, we provide a key estimate, formulated in Lemma~\ref{L:optimcutoff} below, that is central in our proof of
Theorem~\ref{T1}. Before we come to this lemma, we recall suitable versions of the Sobolev inequality
\begin{theorem}
 Fix $d\geq2$. For every $s\in[1,d)$ set $s_d^*:=\frac{ds}{d-s}$.
 \begin{enumerate}[(i)]
  \item For every $s\in[1,d)$ there exists $c=c(d,s)\in[1,\infty)$ such that for every $f:\mathbb Z^d\to\R$ it holds
 \begin{equation}\label{est:sobolev:bulk}
 \|f-(f)_{B(n)}\|_{L^{s_d^*}(B(n))}\leq c\|\nabla f\|_{L^s(B(n))},
 \end{equation}
 where $(f)_{B(n)}:=\frac1{|B(n)|}\sum_{x\in B(n)}f(x)$.
 \item For every $s\in[1,d-1)$ there exists $c=c(d,s)\in[1,\infty)$ such that for every $f:\mathbb Z^d\to\R$ it holds
 \begin{equation}\label{est:sobolev:sphere}
 \|f\|_{L^{s_{d-1}^*}(\partial B(n))}\leq c(\|\nabla f\|_{L^s(\partial B(n))}+n^{-1}\| f\|_{L^s(\partial B(n))}).
 \end{equation}

 \end{enumerate}
 \end{theorem}

\begin{proof}
The above statements are standard. Since we did not find a textbook reference for the discrete situation considered here we provide the argument for some parts of the statement. In what follows we write $\lesssim$ if $\leq$ holds up to a positive constant that depends only on the dimension $d$. 

\step 1 Proof of part~(i). For $s\in(1,d)$ the proof of the claim can be found in \cite[Theorem~2.6]{MO16}. It is left to consider the case $s=1$. In \cite{BR18} it is proven that
\begin{align}
\forall f:\mathbb Z^d \to\R\mbox{ with finite support}\quad \biggl(\sum_{x\in\mathbb Z^d}|f(x)|^\frac{d}{d-1}\biggr)^\frac{d-1}d\lesssim &\sum_{\ee\in\mathbb B^d}|\nabla f(\ee)|,\label{sobolev:1:a}\\
\|f-(f)_{B(n)}\|_{L^{1}(B(n))}\lesssim& |B(n)|^\frac1d \|\nabla f\|_{L^1(B(n))}\label{sobolev:1:b},
\end{align}
see \cite[Lemma~2.1 and 2.2]{BR18}. We deduce \eqref{est:sobolev:bulk} (with $s=1$) from \eqref{sobolev:1:a} and \eqref{sobolev:1:b} by a simple extension argument. Indeed, functions defined on a box $B(n)$ can easily extended by successive reflections (see e.g.\ \cite[Section~9.2]{Brezis}). In particular, there exists $k=k(d)\in\mathbb N\setminus\{1,2\}$ such that for every $g:B(n)\to \R$ we find $\bar g:B(kn)\to\R$ such that
\begin{equation}\label{est:extension}
\bar g(x)=g(x)\quad \mbox{in $B(n)$},\quad \|\bar g\|_{L^1(B(k n))}\lesssim \| g\|_{L^1(B( n))},\quad \|\nabla \bar g\|_{L^1(B(kn))}\lesssim \|\nabla g\|_{L^1(B(n))}.
\end{equation}
Choose $g:=f-(f)_{B(n)}$ and consider a cut-off function 
\begin{equation}\label{eta}
\eta:\mathbb Z^d\to[0,1],\quad \eta=1\quad \mbox{in $ B(n)$}\quad\eta=0\quad\mbox{in $\mathbb Z^d\setminus B(3n-1)$,}\quad |\nabla \eta(\ee)|\lesssim n^{-1}\quad \mbox{for all $\ee\in\mathbb B^d$}.
\end{equation}
Then,
\begin{eqnarray*}
 \|f-(f)_{B(n)}\|_{L^{\frac{d}{d-1}}(B(n))}&\stackrel{\eqref{est:extension}\eqref{eta}}{\leq}& \|\eta\bar g\|_{L^{\frac{d}{d-1}}(B(kn))}\\
 &\stackrel{\eqref{sobolev:1:a}}{\lesssim}&  \|\nabla (\eta\bar g)\|_{L^1(B(kn))}\\
 &\stackrel{\eqref{chainrule}\eqref{eta}}{\lesssim}& \|\nabla \bar g\|_{L^1(B(kn))}+n^{-1}\|\bar g\|_{L^1(B(kn))}\\
 &\stackrel{\eqref{est:extension}}{\lesssim}& \|\nabla f\|_{L^1(B(n))}+n^{-1}\|f-(f)_{B(n)}\|_{L^1(B(n))}\\
 &\stackrel{\eqref{sobolev:1:b}}{\lesssim}& \|\nabla f\|_{L^1(B(n))},
\end{eqnarray*}
where in the last estimate we used $|B(n)|^\frac1d\lesssim n$.

\step 2 Proof of part~(ii). Consider a facet $F$ of $\partial B(n)$ given by $\{x\in B(n)\,|\,x\cdot e_j=t n\}$ for some $j\in\{1,\dots,d\}$ and $t\in\{-1,1\}$. Then, appealing to part~(i), we find $c=c(d,s)\in[1,\infty)$ such that
 \begin{align}\label{est:proof:sobolev:sphere}
 \|f\|_{L^{s_{d-1}^*}(F)}\leq& \|f-(f)_{F}\|_{L^{s_{d-1}^*}(F)}+\|(f)_{F}\|_{L^{s_{d-1}^*}(F)}\notag\\
 \leq&c\|\nabla f\|_{L^s(F)}+|F|^{\frac1s-\frac1{d-1}}|(f)_{F}|\notag\\
 \leq&c\|\nabla f\|_{L^s(F)}+|F|^{-\frac1{d-1}}\|f\|_{L^s(F)}.
 \end{align}
Summing \eqref{est:proof:sobolev:sphere} over all facets $F$ and using $|F|=(2n-1)^{d-1}$, we obtain \eqref{est:sobolev:sphere}. 
\end{proof}

\begin{lemma}\label{L:optimcutoff}
Fix $d\geq3$, $\omega\in\Omega$, $\rho,\sigma \in \mathbb N$ with $\rho<\sigma$ and $v:\mathbb Z^d \to\mathbb R$.  Consider
\begin{equation*}
 J(\rho,\sigma,v):=\inf\biggl\{\sum_{\ee \in \mathbb B^d}\omega(\ee) (|v|(\ee))^2(\nabla \eta(\ee))^2 \;|\;\eta:\mathbb Z^d\to[0,\infty),\, \mbox{$\eta=1$ in $B(\rho)$ and $\eta=0$ in $\mathbb Z^d\setminus B(\sigma-1)$}\biggr\},
\end{equation*}
where $|v|(\ee)=\frac12 (|v(\underline \ee)|+|v(\overline \ee)|)$. Then there exists $c=c(d,p)\in[1,\infty)$ such that
\begin{equation}\label{est:cutoff}
J(\rho,\sigma,v)\leq c(\sigma-\rho)^{-\frac{2d}{d-1}}\|\omega\|_{L^p(B(\sigma)\setminus B(\rho-1))}\left(\|\nabla v\|_{L^{p_*}(B(\sigma)\setminus B(\rho-1))}^2+\rho^{-2}\|v\|_{L^{p_*}(B(\sigma)\setminus B(\rho-1))}^2\right),
\end{equation}
where $p_*$ is given by $\frac{1}{p_*}=\frac12-\frac1{2p}+\frac{1}{d-1}$. 
\end{lemma}

\begin{proof}[Proof of Lemma~\ref{L:optimcutoff}]
\step 1 We claim 
\begin{equation}\label{1dmin}
 J(\rho,\sigma,v)\leq  (\sigma-\rho)^{-(1+\frac1\gamma)} \biggl(\sum_{k=\rho}^{\sigma-1} \biggl(\sum_{\ee\in S(k)} \omega(\ee) |v|(\ee)^2\biggr)^\gamma\biggr)^\frac1\gamma\qquad\mbox{for every $\gamma>0$},
\end{equation}
where for every $m\in\mathbb N$
\begin{equation*}
S(m):=\{\ee\in\mathbb B^d\,|\,\underline \ee\in \partial B(m),\, \overline \ee \in \partial B(m+1)\}.
\end{equation*}
Restricting the class of admissible cut-off functions to those of the form $\eta(x)=\hat \eta(\max_{i=1,\dots,d}\{|x\cdot e_i|\})$, we obtain
\begin{equation}\label{1dmin1}
 J(\rho,\sigma,v)\leq \inf\biggl\{\sum_{k=\rho}^{\sigma-1} \hat \eta'(k)^2\biggl(\sum_{\ee\in S(k)} \omega(\ee) (|v|(\ee))^2\biggr) \;|\;\hat\eta:\mathbb N\to[0,\infty),\,\hat\eta(\rho)=1,\,\hat\eta(\sigma)=0\biggr\}=:J_{\rm 1d},
\end{equation}
where $\hat\eta'(k):=\hat\eta(k+1)-\hat\eta(k)$. The minimization problem \eqref{1dmin1} can be solved explicitly. Indeed, set $f(k):=\sum_{\ee\in S(k)} \omega(\ee) (|v|(\ee))^2$ for every $k\in\mathbb Z$ and suppose $f(k)>0$ for every $k\in\{\rho,\dots,\sigma-1\}$. Then, $\hat\eta:\mathbb N\to[0,\infty)$ with $\hat\eta(i)=1-\left(\sum_{k=\rho}^{\sigma-1}f(k)^{-1}\right)^{-1}\sum_{k=\rho}^{i-1}f(k)^{-1}$ is a valid competitor in the minimization problem for $J_{\rm 1d}$ and we obtain
\begin{equation*}
 J(\rho,\sigma,v)\leq \biggl(\sum_{k=\rho}^{\sigma-1}\biggl(\sum_{\ee\in S(k)} \omega(\ee) (|v|(\ee))^2\biggr)^{-1}\biggr)^{-1}.
\end{equation*}
By H\"older inequality, we obtain for any $s>1$ that $\sigma-\rho=\sum_{k=\rho}^{\sigma-1} (\frac{f}{f})^{\frac1{s'}}\leq\left(\sum_{k=\rho}^{\sigma-1} f^\frac{s}{s'}\right)^\frac1s\left(\sum_{k=\rho}^{\sigma-1} \frac1{f}\right)^\frac1{s'}$ with $s'=\frac{s}{s-1}$ and thus
\begin{equation*}
J_{\rm 1d}\leq (\sigma-\rho)^{-s'}\biggl(\sum_{k=\rho}^{\sigma-1} \biggl(\sum_{\ee\in S(k)} \omega(\ee) (|v|(\ee))^2\biggr)^\frac{s}{s'}\biggr)^{\frac{s'}s}.
\end{equation*}
The claim \eqref{1dmin} follows with $\gamma=s-1>0$. Finally, if $f(k)=\sum_{\ee\in S(k)} \omega(\ee) (|v|(\ee))^2=0$ for some $k\in\{\rho,\dots,\sigma-1\}$, we easily obtain $J_{\rm 1d}=0$ and \eqref{1dmin} is trivially satisfied.

\step 2 We estimate the right-hand side of \eqref{1dmin} with the help of the H\"older inequality and Sobolev inequality in the form \eqref{est:sobolev:sphere}. More precisely, there exists $c=c(p,d)\in[1,\infty)$ (changing from line to line) such that
\begin{eqnarray*}
J(\rho,\sigma,v)&\leq&  \frac{c}{(\sigma-\rho)^{1+\frac1\gamma}} \biggl(\sum_{k=\rho}^{\sigma-1}  \biggl(\sum_{\ee \in S(k)} \omega(\ee)^p\biggr)^{\frac\gamma{p}}\biggl(\sum_{x\in \partial B(k)} |v(x)|^{\frac{2p}{p-1}}+\sum_{x\in \partial B(k+1)} |v(x)|^{\frac{2p}{p-1}}\biggr)^{\frac{(p-1)\gamma}{p}}\biggr)^\frac1\gamma\notag\\
&\leq&  \frac{2^{1-\frac1p}c}{(\sigma-\rho)^{1+\frac1\gamma}} \biggl(\sum_{k=\rho}^{\sigma-1}  \biggl(\sum_{\ee \in S(k)} \omega(\ee)^p\biggr)^{\frac\gamma{p}}\biggl(\|v\|_{L^{\frac{2p}{p-1}}(\partial B(k))}^{2\gamma}+\|v\|_{L^{\frac{2p}{p-1}}(\partial B(k+1))}^{2\gamma}\biggr)\biggr)^\frac1\gamma\notag\\
&\stackrel{\eqref{est:sobolev:sphere}}{\leq}& \frac{c}{(\sigma-\rho)^{1+\frac1\gamma}} \biggl(\sum_{n=\rho}^{\sigma-1}  \biggl(\sum_{\ee\in S(k)} \omega(\ee)^p\biggr)^{\frac\gamma{p}}\sum_{i=k}^{k+1}\biggl(\|\nabla v\|_{L^{p_*}(\partial B(i))}^{2\gamma}+i^{-2\gamma}\|v\|_{L^{p_*}(\partial B(i))}^{2\gamma}\biggr)\biggr)^\frac1\gamma,
\end{eqnarray*}
where $\frac12-\frac1{2p}=\frac1{p_*}-\frac1{d-1}$ (note that $(p_*)_{d-1}^*=\frac{2p}{p-1}$).  The choice $\gamma=\frac{d-1}{d+1}$ yields $\frac{\gamma}{p}+\frac{2\gamma}{{p_*}}=1$ and thus by H\"olders inequality we obtain \eqref{est:cutoff} for some $c=c(d,p)\in[1,\infty)$.
\end{proof}

\subsection{Proof of Theorem~\ref{T1}}

We first present a weaker version of Theorem~\ref{T1} in which the right-hand side of the estimate \eqref{est:T:boundl1} is replaced by a slightly larger term. 
\begin{theorem}\label{T:bound}
Fix $d\geq3$, $\omega\in\Omega$ and let $p,q\in(1,\infty]$ be such that $\frac1p+\frac1q<\frac2{d-1}$. Then there exists $c=c(d,p,q)\in[1,\infty)$ such that solutions of $\nabla^*({\omega}\nabla u)=0$ in $\mathbb Z^d$ satisfy for every $n\in \mathbb N$
\begin{equation*}
 \max_{x\in B(n)}|u(x)|\leq c\Lambda^\omega(B(4n))^{\frac{\delta+1}{2\delta}}\|u\|_{\underline L^{2p'}(B(4n))},
\end{equation*}
where $\delta=\frac{1}{d-1}-\frac1{2p}-\frac1{2q}>0$, $p'=\frac{p}{p-1}$ and  $\Lambda^\omega$ is defined in \eqref{def:lambda}.
\end{theorem}

\begin{proof}[Proof of Theorem~\ref{T:bound}]

Throughout the proof we write $\lesssim$ if $\leq$ holds up to a positive constant that depends only on $d,p$ and $q$. For a function $v:\mathbb Z^d$ and $\alpha\geq1$, we set
\begin{equation*}
\tilde v_\alpha:=({\rm sign}\; v)|v|^\alpha.
\end{equation*}

\step 1 Basic energy estimate.

We claim that for every $\alpha\geq1$ and $n,\rho,\sigma\in\mathbb N$ with $n\leq \rho<\sigma\leq 2n$ it holds
\begin{align}\label{eq:moser:s0}
 \|\nabla \tilde u_\alpha\|_{\underline L^\frac{2q}{q+1}( B(\rho))}\lesssim& \alpha\frac{\Lambda^\omega(B(2n))^\frac12}{(1-\frac{\rho}\sigma)^\frac{d}{d-1} }\left(\rho^{-1}\|\tilde u_\alpha\|_{\underline L^{p_*}(B(\sigma))}+\|\nabla \tilde u_\alpha\|_{\underline L^{p_*}( B(\sigma))}\right),\\
 \|\nabla  u\|_{\underline L^\frac{2q}{q+1}( B(\rho))}\lesssim&\frac{\Lambda^\omega(B(2n))^\frac12}{\sigma-\rho} \|u\|_{\underline L^{2p'}(B(\sigma))},\label{eq:moser:s0b}
\end{align}
where $\frac1{p_*}=\frac12-\frac1{2p}+\frac1{d-1}$ and $p'=\frac{p}{p-1}$.

\substep{1.1} Let $\eta:\mathbb Z^d\to[0,\infty)$ be such that $\eta=0$ in $\mathbb Z^d\setminus B(2n-1)$. We claim that for every $\alpha\geq1$
\begin{equation}\label{eq:moser:s1}
 \sum_{{\rm e\in \mathbb B^d}}\eta^2(\ee)\omega(\ee)\nabla \tilde u_\alpha(\ee)^2\leq \frac{256\alpha^4}{(2\alpha-1)^2}\sum_{\ee\in \mathbb B^d}\omega(\ee)(|u^\alpha|(\ee))^2(\nabla \eta(\ee))^2,
\end{equation}
where $\eta^2(\ee)=\frac12 (\eta^2(\overline \ee)+\eta^2(\underline \ee))$ and $|u^\alpha|(\ee)=\frac12(|u(\overline \ee)|^\alpha+|u(\underline \ee)|^\alpha)$, see Section~\ref{sec:notation}. Using $\mathcal L^\omega u\stackrel{\eqref{Lomegadiv}}{=}-\nabla^*(\omega\nabla u)=0$ in $\mathbb Z^d$ and the summation by parts formula \eqref{sumbyparts} with $F=\omega\nabla u$ and $f=\eta^2\tilde u_{2\alpha-1}$, we obtain
\begin{align}\label{eq:moser:s1:1}
0=& \sum_{\ee\in\mathbb B^d} \omega(\ee)\nabla u(\ee)\nabla (\eta^2\tilde u_{2\alpha-1})(\ee)\notag\\
 =&\sum_{\ee\in \mathbb B^d}2\eta(\ee)\tilde u_{2\alpha-1}(\ee)\omega(\ee)\nabla u(\ee)\nabla \eta(\ee)+\sum_{\ee\in \mathbb B^d}\eta^2(\ee)\omega(\ee)\nabla u(\ee)\nabla \tilde u_{2\alpha-1}(\ee),
\end{align}
where we use the discrete chain rule \eqref{chainrule} and $\nabla \eta^2(\ee)=\eta^2(\overline \ee)-\eta^2(\underline \ee)=(\eta(\overline \ee)-\eta(\underline \ee))(\eta(\overline \ee)+\eta(\underline \ee))=2\nabla \eta(\ee)\eta(\ee)$. Estimate \eqref{est:d:chain:2} implies $\nabla \tilde u_\alpha(\ee)^2\leq \frac{\alpha^2}{2\alpha-1}(\nabla u(\ee)\nabla \tilde u_{2\alpha-1}(\ee))$ for all $\ee\in\mathbb B^d$ and thus
\begin{equation}\label{eq:moser:s1:2}
\sum_{\ee\in\mathbb B^d}\eta^2(\ee)\omega(\ee)\nabla u(\ee)\nabla \tilde u_{2\alpha-1}(\ee)\geq \frac{2\alpha-1}{\alpha^2}\sum_{\ee\in\mathbb B^d}\eta^2(\ee)\omega(\ee)\nabla \tilde u_\alpha(\ee)^2.
\end{equation}
To estimate the second term, we use the pointwise inequality $|\tilde u_{2\alpha-1}|(\ee) |\nabla u(\ee)|\leq 8|\nabla \tilde u_\alpha(\ee)|u|^\alpha(\ee)$ (see \eqref{est:d:chain:3}) and thus by Young's inequality (together with elementary inequality $\eta(\ee)^2\leq\eta^2(\ee)$)
\begin{align}\label{eq:moser:s1:3}
2\sum_{\ee\in \mathbb B^d}\omega(\ee)\nabla u(\ee)\eta(\ee)\nabla \eta(\ee)\tilde u_{2\alpha-1}(\ee)\leq&\frac{2\alpha-1}{2\alpha^2}\sum_{\ee\in\mathbb B^d}\omega(\ee)\eta^2(\ee)(\nabla \tilde u_\alpha(\ee))^2\notag\\
&+\frac{128 \alpha^2}{2\alpha-1}\sum_{\ee\in \mathbb B^d}\omega(\ee)(|u^\alpha|(\ee))^2(\nabla \eta(\ee))^2.
\end{align}
Combining \eqref{eq:moser:s1:1}--\eqref{eq:moser:s1:3}, we obtain \eqref{eq:moser:s1}.

\substep{1.2} Proof of \eqref{eq:moser:s0}. By minimizing the right-hand side of \eqref{eq:moser:s1} over all $\eta:\mathbb Z^d\to[0,\infty)$ satisfying $\eta=1$ on $B(\rho)$ and $\eta=0$ in $\mathbb Z^d\setminus B(\sigma-1)$, we obtain in view of Lemma~\ref{L:optimcutoff}
\begin{equation*}
\sum_{{\rm e\in B(\rho)}}\omega(\ee)(\nabla \tilde u_\alpha(\ee))^2\lesssim\alpha^2(\sigma-\rho)^{-\frac{2d}{d-1}}\|\omega\|_{L^p(B(\sigma))}\left(\|\nabla \tilde u_\alpha\|_{L^{p_*}(B(\sigma))}^2+\rho^{-2}\|\tilde u_\alpha\|_{L^{p_*}(B(\sigma))}^2\right).
\end{equation*}

By H\"older inequality, we obtain
\begin{eqnarray*}
 \|\nabla \tilde u_\alpha\|_{L^{\frac{2q}{q+1}}(B(\rho))}^2&\leq&  \|\omega^{-1}\|_{L^q(B(\rho))}\biggl(\sum_{\ee \in B(\rho)}\omega(\ee)(\nabla \tilde u_\alpha (\ee))^2\biggr)\\
 &\lesssim&\alpha^2(\sigma-\rho)^{-\frac{2d}{d-1}}\|\omega^{-1}\|_{L^q(B(\rho))}\|\omega\|_{L^p(B(\sigma))}\left(\|\nabla \tilde u_\alpha\|_{L^{p_*}(B(\sigma))}^2+\rho^{-2}\|\tilde u_\alpha\|_{L^{p_*}(B(\sigma))}^2\right)
\end{eqnarray*}
and the claim \eqref{eq:moser:s0} follows.

\substep{1.3} Estimate \eqref{eq:moser:s0b} is a straightforward consequence of \eqref{eq:moser:s1} (with $\alpha=1$ and a 'linear' cut-off function $\eta$ satisfying $\eta(x)=1$ for $x\in B(\rho)$, $\eta=0$ on $\mathbb Z^d\setminus B(\sigma-1)$ and $\nabla \eta(\ee)=(\sigma - \rho)^{-1}$ for all $\ee \in S(k)$ and $k\in\{\rho,\dots,\sigma-1\}$) and an application of H\"older's inequality.

\step 2 One-step improvement.

Fix $\alpha\geq1$ and $\rho,\sigma\in\mathbb N$ with $n\leq \rho<\sigma\leq 2n$. We claim that there exists $c=c(d,p,q)\in[1,\infty)$ such that
\begin{align}
 &\|\tilde u_{\chi\alpha}\|_{\underline W^{1,p_*}(B(\rho))}^\frac1{\chi\alpha} \leq \left(\frac{c \alpha\Lambda^\omega( B(2n))^\frac12}{(1-\frac\rho\sigma)^{\frac{d}{d-1}}} \right)^\frac1{\alpha}\|\tilde u_\alpha\|_{\underline W^{1,p_*}(B(\sigma))}^\frac1{\alpha},\label{eq:moser:iterate}
\end{align}

where $\chi:=1+\delta>1$ with $\delta=\frac1{d-1}-\frac1{2q}-\frac1{2p}>0$ and for $p\in[1,\infty)$ and any pair $(y,n)\in\mathbb Z^d\times \mathbb N$ the (normalized) Sobolev norm $\|\cdot\|_{\underline W^{1,p}(B(y,n))}$ is defined for any $v:\mathbb Z^d\to\R$ as
\begin{equation*}
\|v\|_{\underline W^{1,p}(B(y,n))}:=\|v\|_{\underline L^p(B(y,n))}+n\|\nabla v\|_{\underline L^p(B(y,n))}.
\end{equation*}
In order to establish \eqref{eq:moser:iterate}, we use Step~1 and the following two estimates
\begin{align}
 \|\tilde u_{\chi\alpha}\|_{\underline L^{p_*}(B(\rho))}=&\|\tilde u_{\alpha}\|_{\underline L^{\chi p_*}(B(\rho))}^\chi\lesssim \left(\|\tilde u_{\alpha}\|_{\underline L^{p_*}(B(\rho))}+\rho\|\nabla \tilde u_{\alpha}\|_{\underline L^{p_*}(B(\rho))}\right)^\chi\label{eq:moser:iterate:2}\\
 \rho\|\nabla \tilde u_{\chi\alpha}\|_{\underline L^{p_*}(B(\rho))}\lesssim&\rho\|\nabla \tilde u_{\alpha}\|_{\underline L^\frac{2q}{q+1}(B(\rho))}\|u^\alpha\|_{\underline L^1(B(\rho))}^\delta\lesssim (\rho\|\nabla \tilde u_{\alpha}\|_{\underline L^\frac{2q}{q+1}(B(\rho))})^\chi+\|u^\alpha\|_{\underline L^1(B(\rho))}^\chi.\label{eq:moser:iterate:1}
\end{align}
Estimate \eqref{eq:moser:iterate:2} is a consequence of Sobolev inequality (note that $\chi\in(1,\frac{d}{d-1}]$) and \eqref{eq:moser:iterate:1} follows from
\begin{align*}
 \|\nabla \tilde u_{\alpha(1+\delta)}\|_{L^{p_*}(B(\rho))}\stackrel{\eqref{est:d:chain:1}}{\leq}& (1+\delta)\biggl(\sum_{\ee\in B(\rho)}|\nabla \tilde u_\alpha(\ee)|^{p_*}(2|u|^{\alpha \delta}(\ee))^{p_*}\biggr)^\frac1{p_*}\\
 \leq&(1+\delta)\|\nabla \tilde u_{\alpha}\|_{L^\frac{2q}{q+1}(B(\rho))}\biggl(\sum_{\ee\in B(\rho)}(2|u|^{\alpha \delta}(\ee))^{\frac1\delta}\biggr)^\delta
\end{align*}
and an application of Young's inequality with exponent $\chi$ and $\frac{\chi}{\chi-1}=\frac{\chi}\delta$ (recall $\chi=1+\delta$). \\
Appealing to estimates \eqref{eq:moser:s0}, \eqref{eq:moser:iterate:1}, \eqref{eq:moser:iterate:2} and Jensen inequality in the form $\|\cdot\|_{\underline L^1(B(\rho))}\leq \|\cdot\|_{\underline L^{p_*}(B(\rho))}$, we obtain
\begin{eqnarray}\label{eq:moser:iterate:3}
  &&\|\tilde u_{\chi\alpha}\|_{\underline L^{p_*}(B(\rho))}+\rho\|\nabla \tilde u_{\chi\alpha}\|_{\underline L^{p_*}(B(\rho))}\notag\\
 &\stackrel{\eqref{eq:moser:iterate:1},\eqref{eq:moser:iterate:2}}{\lesssim}&\left(\|\tilde u_{\alpha}\|_{\underline L^{p_*}(B(\rho))}+\rho\|\nabla \tilde u_{\alpha}\|_{\underline L^{p_*}(B(\rho))}\right)^\chi+(\rho\|\nabla \tilde u_{\alpha}\|_{\underline L^\frac{2q}{q+1}(B(\rho))})^\chi\notag\\
 &\stackrel{\eqref{eq:moser:s0}}{\lesssim}&\left(\alpha\left(1-\frac{\rho}\sigma\right)^{-\frac{d}{d-1}} \Lambda^\omega( B(2n))^\frac12\right)^\chi\left(\|\tilde u_\alpha\|_{\underline L^{p_*}(B(\sigma))}+\rho\|\nabla \tilde u_\alpha\|_{\underline L^{p_*}(B(\sigma))}\right)^\chi,
\end{eqnarray}
where we used for the last estimate that $\alpha\geq1$, $0<\rho<\sigma$ and $\Lambda^\omega( B(2n))\geq1$. Clearly \eqref{eq:moser:iterate:3} implies the claimed estimate \eqref{eq:moser:iterate}

\step 3 Iteration.

For $\nu\in\mathbb N\cup \{0\}$, set $\alpha_\nu=\chi^{\nu-1}$ and $\rho_\nu=n+\lfloor \frac{n}{2^\nu}\rfloor$. Then for any $\nu\in\mathbb N$ satisfying $2^\nu\leq n$, estimate \eqref{eq:moser:iterate} (with $\alpha=\alpha_\nu$, $\rho=\rho_\nu$ and $\sigma=\rho_{\nu-1}$) implies that there exists $c=c(d,p,q)\in[1,\infty)$ such that
\begin{equation}\label{est:moser:iterate}
\|\tilde u_{\chi^\nu}\|_{\underline W^{1,p_*}(B(\rho_\nu))}^\frac1{\chi^\nu} \leq \left(c \Lambda^\omega( B(2n))^\frac12 (4\chi)^{\nu} \right)^\frac1{\chi^{\nu-1}}\|\tilde u_{\chi^{\nu-1}}\|_{\underline W^{1,p_*}(B(\rho_{\nu-1}))}^\frac1{\chi^{\nu-1}},
\end{equation}
where we used the elementary estimate
\begin{align*}
\left(1-\frac{\rho_{\nu}}{\rho_{\nu-1}}\right)^{-\frac{d}{d-1}}=\left(\frac{n+\lfloor\tfrac{n}{2^{\nu-1}}\rfloor}{\lfloor\tfrac{n}{2^{\nu-1}}\rfloor-\lfloor\tfrac{n}{2^{\nu}}\rfloor}\right)^{\frac{d}{d-1}}\leq \left(2^{2+\nu}\right)^\frac{d}{d-1}\stackrel{d\geq2}{\leq}  4^{2+\nu}.
\end{align*}
Set $\hat\nu(n):=\max\{\nu\in\mathbb N\,|\,2^\nu\leq n\}$. Using \eqref{est:moser:iterate} $\hat \nu(n)$-times, we obtain
\begin{align}\label{est:moser:almostfinal}
\| u\|_{\underline L^{p_*\chi^{\hat \nu(n)}}(B(n))}\leq& \prod_{\nu=1}^{\hat \nu(n)} \left(c \Lambda^\omega( B(2n))^\frac12(4\chi)^\nu \right)^{\frac1{\chi^{\nu-1}}}\|u\|_{\underline W^{1,p_*}(B(2n))}\notag\\
 \leq&\left(c\Lambda^\omega( B(2n))^\frac12\right)^{\sum_{\nu=0}^\infty \frac1{\chi^{\nu}}}\left(4\chi\right)^{\sum_{\nu=0}^\infty \frac{\nu+1}{\chi^{\nu}}}\| u\|_{\underline W^{1,p_*}(B(2n))}.
\end{align}
To estimate the right-hand side of \eqref{est:moser:almostfinal}, we use \eqref{eq:moser:s0b}, Jensen's inequality and the fact that $p_*<\frac{2q}{q+1}\leq 2\leq 2p'$ 
\begin{align*}
\|\nabla u\|_{\underline L^{p_*}(B(2n))}\leq \|\nabla u\|_{\underline L^{\frac{2q}{q+1}}(B(2n))}\stackrel{\eqref{eq:moser:s0b}}{\lesssim} n^{-1}\Lambda^\omega( B(4n))^\frac12 \|u \|_{\underline L^{2p'}(B(4n))},\quad \|u\|_{L^{p_*}(B(2n))}\lesssim\|u\|_{\underline L^{2p'}(B(4n))}.
\end{align*}
Since $\Lambda^\omega\geq1$ and $\sum_{\nu=0}^\infty (1+\nu) \chi^{-\nu}\lesssim1$, we obtain 
\begin{align*}
 \|u\|_{L^\infty(B(n))}\leq& |B(n)|^\frac1{p_*\chi^{\hat \nu(n)}}\| u\|_{\underline L^{p_*\chi^{\hat \nu(n)}}(B(n))}\\
 \lesssim& |B(n)|^\frac1{p_*\chi^{\hat \nu(n)}} \Lambda^\omega( B(2n))^{\frac{1}{2}(\frac1{1-\chi^{-1}}-1)}\|u\|_{\underline W^{1,p_*}(B(2n))}\\
 \lesssim& |B(n)|^\frac1{p_*\chi^{\hat \nu(n)}} \Lambda^\omega( B(4n))^{\frac{1}{2}\frac{\chi}{\chi-1}}\|u\|_{L^{2 p'}(B(4n))}.
\end{align*}
Hence, it is left to show that $|B(n)|^\frac1{\chi^{\hat \nu(n)}}\lesssim1$ (recall $\chi=1+\delta$). Assuming $n\in\mathbb N$ is sufficiently large, we have $\hat\nu(n)\geq \frac12\log_2 n$ and thus
\begin{align}\label{est:finalvolume}
|B(n)|^\frac1{\chi^{\hat \nu(n)}}\lesssim n^{\frac{d}{\chi^{\frac12 \log_2 n}}}=n^{\frac{d}{n^{\frac12 \log_2 \chi}}}=\exp(\log(n^{\frac{d}{n^{\frac12 \log_2 \chi}}}))=\exp(\frac{d}{n^{\frac12 \log_2 \chi}}\log(n))\lesssim1,
\end{align}
which finishes the proof.

\end{proof}

Using a well-known iteration argument (see e.g.\ \cite[Corollary~3.9]{ADS15}), we refine the statement of Theorem~\ref{T:bound} and obtain
\begin{corollary}\label{C:est:gamma}
Fix $d\geq3$, $\omega\in \Omega$ and let $p,q\in(1,\infty]$ be such that $\frac1p+\frac1q<\frac2{d-1}$. For every $\gamma\in(0,1]$ there exists $c=c(d,p,q,\gamma)\in[1,\infty)$ such that solutions of $\nabla^*({\omega}\nabla u)=0$ in $\mathbb Z^d$ satisfy for every $n\in \mathbb N$
\begin{equation*}
 \max_{x\in B(n)}|u(x)|\leq c\Lambda^\omega( B(2n))^{\frac{\delta+1}{2\delta\gamma}}\|u\|_{\underline L^{2p'\gamma}(B(2n))},
\end{equation*}
where $\delta=\frac{1}{d-1}-\frac1{2p}-\frac1{2q}>0$, $p'=\frac{p}{p-1}$ and $\Lambda^\omega$ is defined in \eqref{def:lambda}.
\end{corollary}

\begin{proof}[Proof of Theorem~\ref{T1}]
 The choice $\gamma=\frac1{2p'}$ in Corollary~\ref{C:est:gamma} yield \eqref{est:T:boundl1} for $y=0\in\mathbb Z^d$ and by translation we obtain the general claim.
\end{proof}

\begin{proof}[Proof of Corollary~\ref{C:est:gamma}]
Throughout the proof we write $\lesssim$ if $\leq$ holds up to a positive constant that depends only on $d,p$ and $q$.

\step 1 We claim that for every $N,N'\in\mathbb N$ with $N'<N$ 
\begin{equation}\label{est:scaling}
 \max_{x\in B(N')}|u(x)|\lesssim \frac{\Lambda^\omega( B(N))^{\frac{\delta+1}{2\delta}}}{(1-\frac{N'}{N})^{s}}\|u\|_{\underline L^\frac{2p}{p-1}(B(N))},
\end{equation}
where $s:=\frac{d}2(1+\frac1q+(\frac1p+\frac1q)\frac1\delta)$. Without loss of generality we assume that $N-N'\geq 4$, since otherwise \eqref{est:scaling} follows from the discrete $L^\infty$-$L^{1}$-estimate. Theorem~\ref{T:bound} and a simple translation argument yield for every $y\in B(N')$ 
\begin{align*}
\max_{x\in B(y,\lfloor \frac{N-N'}4\rfloor)}|u(x)|\lesssim& \Lambda^\omega( B(y,N-N'))^{\frac{\delta+1}{2\delta}}\|u\|_{\underline L^\frac{2p}{p-1}(B(y,N-N'))}\\
\lesssim&\left(\frac{N^d}{(N-N')^d}\right)^{\frac{p-1}{2p}+(\frac1q+\frac1p)\frac{\delta+1}{2\delta}}\Lambda(B(N))^{\frac{\delta+1}{2\delta}}\|u\|_{\underline L^\frac{2p}{p-1}(B(N))}
\end{align*}
and estimate \eqref{est:scaling} follows.

\step 2 Iteration. Fix $\gamma\in(0,1)$. For $\nu\in\mathbb N\cup\{0\}$, we set $\rho_\nu=2n-\lfloor \frac{n}{2^\nu}\rfloor$. Combining the elementary interpolation inequality
\begin{equation}\label{est:interpolate}
\|u\|_{\underline L^{2p'}(B(\rho_\nu))}\leq \|u\|_{\underline L^{2p'\gamma}(B(\rho_\nu))}^\gamma\|u\|_{L^{\infty}(B(\rho_\nu))}^{1-\gamma}
\end{equation}
with the estimate \eqref{est:scaling}, we obtain for every $\nu\in\mathbb N$
\begin{eqnarray}\label{est:refine:iterate}
\|u\|_{L^\infty(B_{\rho_{\nu-1}})}&\stackrel{\eqref{est:scaling}}{\lesssim}& \Lambda^\omega( B(\rho_{\nu}))^\frac{\delta+1}{2\delta}(1-\tfrac{\rho_{\nu-1}}{\rho_{\nu}})^{-s}\|u\|_{\underline L^{2p'}(B(\rho_{\nu}))}\notag\\
&\stackrel{\eqref{est:interpolate}}{\leq}& \Lambda^\omega( B(\rho_{\nu}))^\frac{\delta+1}{2\delta}(1-\tfrac{\rho_{\nu-1}}{\rho_{\nu}})^{-s}\|u\|_{\underline L^{2p'\gamma}(B(\rho_\nu)}^\gamma\|u\|_{L^{\infty}(B(\rho_{\nu}))}^{1-\gamma}\notag\\
&\leq&2^{\nu s}C\|u\|_{\underline L^{2p'\gamma}(B(2n))}^\gamma\|u\|_{L^{\infty}(B(\rho_{\nu}))}^{1-\gamma}
\end{eqnarray}
with $C=c\Lambda^\omega( B(2n))^{\frac{\delta+1}{2\delta}}$ and a suitable constant $c=c(d,p,q)\in[1,\infty)$, where we used for the last estimate $\rho_\nu\geq n$ for all $\nu\in\mathbb N$ and $(1-\frac{\rho_{\nu-1}}{\rho_{\nu}})^{-s}\leq (2^{2+\nu})^s$.

Iterating \eqref{est:refine:iterate} from $\nu=1$ to $\hat\nu(n):=\max\{\nu\in\mathbb N\,|\,2^\nu\leq n\}$, we obtain
\begin{eqnarray*}
\| u\|_{L^{\infty}(B(n))}&=&\| u\|_{L^{\infty}(B(\rho_{0}))}\notag\\
&\stackrel{\eqref{est:refine:iterate}}{\leq}& 4^{s\sum_{\nu=0}^{\hat \nu(n)-1}(\nu+1)(1-\gamma)^\nu}\left(C\|u\|_{\underline L^{2p'\gamma}(B(2n))}^\gamma\right)^{\sum_{\nu=0}^{\hat \nu(n)-1}(1-\gamma)^\nu}\|u\|_{L^{\infty}(B(\rho_{\hat \nu(n)}))}^{(1-\gamma)^{\hat \nu(n)}}.
\end{eqnarray*}
Using $\sum_{\nu=0}^\infty(\nu+1)(1-\gamma)^\nu\lesssim1$, $\sum_{\nu=0}^{\hat \nu(n)-1}(1-\gamma)^\nu=\frac1\gamma(1-(1-\gamma)^{\hat \nu(n)})$, $C\geq1$ and the discrete $L^\infty$-$L^1$-estimate, we obtain
\begin{align*}
\| u\|_{L^{\infty}(B(n))}\leq c\Lambda^\omega( B(2n))^\frac{\delta+1}{2\delta\gamma}\|u\|_{\underline L^{2p'\gamma}(B(2n))}|B(n)|^{\frac{(1-\gamma)^{\hat \nu(n)}}{2p'\gamma}},
\end{align*}
where $c=c(\gamma,d,p,q)\in[1,\infty)$. Finally, a similar calculation as in \eqref{est:finalvolume} yields $|B(n)|^{\frac{(1-\gamma)^{\hat \nu(n)}}{2p'\gamma}}\lesssim c(\gamma)\in[1,\infty)$, which finishes the proof.
\end{proof}

\appendix

\section{Technical estimates}

We recall some estimates proven in \cite[Lemma A.1]{ADS15} that we used in the proof of Theorem~\ref{T1}.
\begin{lemma}[\cite{ADS15}, Lemma A.1] For $a\in\R$ and $\alpha\in \R\setminus\{0\}$, set $\tilde a_\alpha=|a|^\alpha {\rm sign} a$.
\begin{enumerate}[(i)]
\item For all $a,b\in\R$ and any $\alpha,\beta\neq0$
\begin{equation}\label{est:d:chain:1}
 |\tilde a_\alpha-\tilde b_\alpha|\leq \left(1\vee \left|\frac\alpha\beta\right|\right)|\tilde a_\beta-\tilde b_\beta|(|a|^{\alpha-\beta}+|b|^{\alpha-\beta})
\end{equation}
\item For all $a,b\in\R$ and $\alpha>\frac12$
\begin{equation}\label{est:d:chain:2}
 (\tilde a_\alpha-\tilde b_\alpha)^2\leq \frac{\alpha^2}{2\alpha-1}(a-b)(\tilde a_{2\alpha-1}-\tilde b_{2\alpha-1})
\end{equation}
\item For all $a,b\in\R$ and $\alpha\geq\frac12$
\begin{equation}\label{est:d:chain:3}
 (|a|^{2\alpha-1}+|b|^{2\alpha-1})(a-b)\leq 4|\tilde a_\alpha-\tilde b_\alpha|(|a|^\alpha+|b|^\alpha).
\end{equation}

\end{enumerate}

\end{lemma}

\section{The two-dimensional case}\label{sec:2d}

In two dimensions Biskup \cite{Biskup} proved sublinearity of the corrector and the QFCLT under the minimal moment condition $p=q=1$ in \eqref{ass:moment}. The reasoning in \cite{Biskup} (which has its origins in \cite{BB07}) combines geometric, analytical and probabilistic arguments. In this section, we sketch a proof of Biskups result that relies only on deterministic regularity theory and the spatial ergodic theorem. The main ingredient is the following local boundedness result:
\begin{proposition}\label{P2d}
Fix $\omega\in(0,\infty)^{\mathbb B^2}$. Then there exists $c\in[1,\infty)$ such that solutions of $\nabla^*({\omega}\nabla u)=0$ in $\mathbb Z^2$ satisfy for every $n\in \mathbb N$
\begin{equation}\label{est:T:boundl12d}
 \max_{x\in B(n)}|u(x)|\leq c\left(n\|\omega^{-1}\|_{\underline L^1(B(2n))}^\frac12\|\omega (\nabla u)^2\|_{\underline L^1(B(2n))}^\frac12+\|u\|_{\underline L^{1}(B(2n))}\right).
\end{equation}

\end{proposition}

\begin{proof}[Proof of Proposition~\ref{P2d}]
Throughout the proof we write $\lesssim$ if $\leq$ holds up to a generic positive constant.

The proof is elementary and relies on three ingredients: First, $\nabla^*({\omega}\nabla u)=0$ in $\mathbb Z^2$ implies a maximum principle in the form
\begin{equation}\label{est:P2d:1}
 \max_{x\in B(n)}|u(x)|\leq  \max_{x\in B(k)}|u(x)|\leq \max_{x\in \partial B(k)}|u(x)|\qquad\mbox{for all  $k\in\{n,\dots,2n\}$}.
\end{equation}
Secondly, since
$$
 \sum_{k=n}^{2n}\left(\|\nabla u\|_{L^1(\partial B(k))}+\frac1n\|u\|_{L^1(\partial B(k))}\right)\leq \|\nabla u\|_{L^1(B(2n))}+\frac1n\|u\|_{L^1(B(2n))},
$$
we can choose a 'good' $\tilde k\in\{n,\dots,2n\}$ satisfying
\begin{equation}\label{est:P2d:2}
 \|\nabla u\|_{L^1(\partial B(\tilde k))}+\frac1n\|u\|_{L^1(\partial B(\tilde k))}\leq \frac1n\left(\|\nabla u\|_{L^1(B(2n))}+\frac1n\|u\|_{L^1(B(2n))}\right).
\end{equation}
The last ingredient is a one-dimensional Sobolev inequality (which follows simply by the discrete version of the fundamental theorem of calculus)
\begin{equation}\label{est:P2d:3}
 \max_{x\in \partial B(k)}|u(x)|\lesssim \|\nabla u\|_{L^1(\partial B( k))}+\frac1k\|u\|_{L^1(\partial B(k))}\qquad\mbox{for all  $k\in\{n,\dots,2n\}$}.
\end{equation}
Combining \eqref{est:P2d:1}-\eqref{est:P2d:3} and $\tilde k\in\{n,\dots,2n\}$, we obtain,
\begin{align*}
\max_{x\in B(n)}|u(x)|\stackrel{\eqref{est:P2d:1}}{\leq}  \max_{x\in \partial B(\tilde k)}|u(x)|\stackrel{\eqref{est:P2d:3}}{\lesssim}  \|\nabla u\|_{L^1(\partial B(\tilde k))}+\frac1n\|u\|_{L^1(\partial B(\tilde k))}\stackrel{\eqref{est:P2d:2}}{\lesssim} n\|\nabla u\|_{\underline L^1(B(2n))}+\|u\|_{\underline L^1(B(2n))},
\end{align*}
where we used in the last inequality also the fact $|B(2n)|\lesssim n^2$. Clearly, \eqref{est:T:boundl12d} follows from the last displayed formula and H\"older's inequality.
\end{proof}

\begin{proposition}\label{P:sublinlinfty:2d}
Let $d=2$ and suppose that part~(i) and (ii) of Assumptions~\ref{ass} are satisfied. Moreover, suppose that $\mathbb E[\omega(\ee)]<\infty$ and $\mathbb E[\omega(\ee)^{-1}]<\infty$  for every $\ee\in\mathbb B^2$. Then, for every $j\in\{1,\dots,d\}$ 
\begin{equation*}
\lim_{n\to\infty}\frac1n\max_{x\in B(n)}|\chi_j(\omega,x)|=0\qquad\mbox{$\mathbb P$-a.s.}
\end{equation*}
\end{proposition}

\begin{proof}[Proof of Proposition~\ref{P:sublinlinfty:2d}]
Throughout the proof we write $\lesssim$ if $\leq$ holds up to a generic positive constant. 

\step 1 More ergodic theory. In contrast to \eqref{est:T:boundl1} the right-hand side of \eqref{est:T:boundl12d} depends on the discrete gradient of the $\mathcal L^\omega$-harmonic function. In the application to the corrector equation $\nabla^*(\omega \nabla \Phi_j)=0$ we use the ergodic theorem to control terms coming from $\nabla \Phi_j$. For this it is convenient to introduce the following measures on $\mathbb Z^d$:
\begin{equation*}
 \iota_j^\omega(x):=\sum_{y\in\mathbb Z^d}\omega(x,y)(\chi_j(\omega,y)-\Pi_j(\omega,y)-(\chi_j(\omega,x)-\Pi_j(\omega,x)))^2\stackrel{\eqref{eq:defchij}}{=}\sum_{y\in\mathbb Z^d}\omega(x,y)(\Phi_j(\omega,y)-\Phi_j(\omega,x))^2,
\end{equation*}
where $j\in\{1,\dots,d\}$ and $\Pi_j$ denotes the position field introduced in Section~\ref{subsec:old}. Since $\Phi_j$ is defined as a projection of $\Pi_j$ on $L^2_{\rm sol}$ in $L^2_{\rm cov}$, we have
\begin{equation}\label{est:iotabound}
\mathbb E[\iota(0)]=\mathbb E\biggl[\sum_{y\in\mathbb Z^d}\omega(0,y)\Phi_j(y)^2\biggr]=\|\Phi_j\|_{L^2_{\rm cov}}^2\leq\|\Pi_j\|_{L^2_{\rm cov}}^2\leq\mathbb E[\mu(0)],
\end{equation}
where we use $\Phi_j(\omega,0)=0$ for every $\omega$ (which follows directly from the cocycle property). Moreover, appealing to the cocycle property of $\Phi_j$, we have $\iota^{\tau_z\omega}(x)=\iota^\omega(x+z)$ for every $x,z\in\mathbb Z^d$ and thus, by \eqref{est:iotabound} and the spatial ergodic theorem, we find a set $\Omega'\in \Omega$ with $\mathbb P[\Omega']=1$ such that for all $\omega\in\Omega'$ and for every $z\in \mathbb Z^d$
\begin{equation}\label{eqgodic:iota}
 \lim_{n\to\infty}\|\iota^\omega\|_{\underline L^1(B(nz,n))}=\mathbb E[\iota^\omega(0)]\leq\mathbb E[\mu(0)]<\infty.
\end{equation}
\step 2 From now on we use the notation of Step~2 in the proof of Proposition~\ref{P:sublinlinfty}. Using estimate \eqref{est:T:boundl12d} instead of \eqref{est:T:boundl1}, we obtain 
\begin{align}\label{moser:ui:2d}
&\|u_j^z\|_{L^\infty(B(\lfloor \frac{n}{m}\rfloor z,\lfloor \frac{n}{m}\rfloor))}\notag\\
\lesssim& \lfloor \tfrac{n}{m}\rfloor\|\omega^{-1}\|_{\underline L^1(B(\lfloor \frac{n}{m}\rfloor z,\lfloor \frac{n}{m}\rfloor))}^\frac12\|\omega (\nabla \chi_j-\nabla \Pi_j)^2\|_{\underline L^1(B(\lfloor \frac{n}{m}\rfloor z,2\lfloor \frac{n}{m}\rfloor))}^\frac12+\|\chi_j\|_{\underline L^1(B(\lfloor \frac{n}{m}\rfloor z,2\lfloor \frac{n}{m}\rfloor))}+\lfloor \tfrac{n}m\rfloor,
\end{align}
instead of \eqref{moser:ui}. Estimate \eqref{moser:ui:2d} implies the following $L^\infty$-estimate on $\chi_j$
\begin{eqnarray}\label{est:phiexi:1:2d}
 & &\|\chi_j\|_{L^\infty(B(n))}\lesssim  \sup_{z\in B(m)}\|u_j^z\|_{L^\infty(B(\lfloor \frac{n}{m}\rfloor z,\lfloor \frac{n}{m}\rfloor))}+\lfloor \tfrac{n}m\rfloor\notag\\
 &\stackrel{\eqref{moser:ui:2d}}{\lesssim}& \sup_{z\in B(m)}\left(\lfloor \tfrac{n}{m}\rfloor\|\omega^{-1}\|_{\underline L^1(B(\lfloor \frac{n}{m}\rfloor z,\lfloor \frac{n}{m}\rfloor))}^\frac12\|\omega (\nabla \Phi_j)^2\|_{\underline L^1(B(\lfloor \frac{n}{m}\rfloor z,2\lfloor \frac{n}{m}\rfloor))}^\frac12+\|\chi_j\|_{\underline L^1(B(\lfloor \frac{n}{m}\rfloor z,2\lfloor \frac{n}{m}\rfloor))}\right)+\lfloor \tfrac{n}m\rfloor\notag\\
 &\lesssim& \lfloor \tfrac{n}{m}\rfloor\sup_{z\in B(m)}\|\omega^{-1}\|_{\underline L^1(B(\lfloor \frac{n}{m}\rfloor z,\lfloor \frac{n}{m}\rfloor))}^\frac12\|\omega (\nabla \Phi_j)^2\|_{\underline L^1(B(\lfloor \frac{n}{m}\rfloor z,2\lfloor \frac{n}{m}\rfloor))}^\frac12+m^{d}\|\chi_j\|_{\underline L^1(B(2n)}+\lfloor \tfrac{n}m\rfloor.
\end{eqnarray}
The ergodic theorem in the versions \eqref{eq:ergodic} and \eqref{eqgodic:iota} implies that $\mathbb P$-a.s.
\begin{align}\label{limsupLambdae2d}
&\limsup_{n\to\infty}\max_{z\in B(m)}\|\omega^{-1}\|_{\underline L^1(B(\lfloor \frac{n}{m}\rfloor z,\lfloor \frac{n}{m}\rfloor))}^\frac12\|\omega (\nabla \Phi_j)^2\|_{\underline L^1(B(\lfloor \frac{n}{m}\rfloor z,2\lfloor \frac{n}{m}\rfloor))}^\frac12\notag\\
\lesssim&  \mathbb E[\nu(0)]^\frac12\mathbb E[\iota_j(0)]^\frac12=\mathbb E[\nu(0)]^\frac12\mathbb E[\mu(0)]^\frac12.
\end{align}
Combining \eqref{est:phiexi:1:2d}, \eqref{limsupLambdae2d} and the $L^1$-sublinearity of $\chi_j$ \eqref{eq:sublinl1}, we obtain
$$
\limsup_{n\to\infty}\frac1n\|\chi_j\|_{L^\infty(B(n))}\lesssim \frac1m(1+\mathbb E[\nu(0)]^\frac12\mathbb E|\mu(0)]^\frac12).
$$
The arbitrariness of $m\in\mathbb N$ yields the desired claim.

\end{proof}

\subsection*{Acknowledgement} The authors were supported by the German Science Foundation DFG in context of the Emmy Noether Junior Research Group BE 5922/1-1.

\end{document}